\documentclass{article}

\usepackage{arxiv}

\usepackage[utf8]{inputenc} % allow utf-8 input
\usepackage[T1]{fontenc}    % use 8-bit T1 fonts
\usepackage{hyperref}       % hyperlinks
\usepackage{url}            % simple URL typesetting
\usepackage{booktabs}       % professional-quality tables
\usepackage{amsfonts}       % blackboard math symbols
\usepackage{nicefrac}       % compact symbols for 1/2, etc.
\usepackage{microtype}      % microtypography
\usepackage{lipsum}		% Can be removed after putting your text content
\usepackage{graphicx}
\usepackage{doi}

\usepackage{amsmath}
\RequirePackage{amssymb}
\RequirePackage{amsthm}
\RequirePackage{cite}
\RequirePackage{bm}
\usepackage{comment}
\newtheorem{thm}{Theorem}
\newtheorem{dfn}[thm]{Definition}
\newtheorem{lem}[thm]{Lemma}
\newtheorem{rmk}{Remark}
\renewcommand{\d}{{\mathrm{d}}}
\newcommand{\R}{{\mathbb{R}}}

\title{Symplecticity of coupled Hamiltonian systems}

\author{Shunpei Terakawa\\
    Graduate School of System Informatics\\
    Kobe University\\
    Kobe, Hyogo, Japan\\
	\texttt{s-terakawa@stu.kobe-u.ac.jp} \\
	\And
	Takaharu Yaguchi\\
	Graduate School of System Informatics\\
    Kobe University\\
    Kobe, Hyogo, Japan\\
	\texttt{yaguchi@pearl.kobe-u.ac.jp} \\
}

\renewcommand{\shorttitle}{}

\begin{document}
\maketitle

\begin{abstract}
	We derived a condition under which a coupled system consisting of two finite-dimensional Hamiltonian systems becomes a Hamiltonian system.
In many cases, an industrial system can be modeled as a coupled system of some subsystems.
Although it is known that symplectic integrators are suitable for discretizing Hamiltonian systems,
the composition of Hamiltonian systems may not be Hamiltonian.
In this paper, 
focusing on a property of Hamiltonian systems, that is, the conservation of the symplectic form, we provide a condition under which two Hamiltonian systems coupled with interactions compose a Hamiltonian system.
\end{abstract}

% keywords can be removed
\keywords{coupled system\and Hamiltonian system\and symplectic integrator\and interaction}

\section{Introduction}
Because many industrial objects are described as coupled systems,
it is important to investigate the properties of systems composed of subsystems that are modeled separately.
For example, in the physical simulation of a piano, it is necessary to consider a model in which the parts described by different governing equations, such as strings, hammers, bridges, and soundboard, are combined by interaction \cite{Chabassier2013}.
In general, numerical simulations are necessary to study such systems;
%In general, such a coupled system is difficult to solve analytically, and numerical methods are indispensable.
however, if the coupled system under investigation is large and/or requires a long-term prediction, it may be difficult to compute numerical solutions with general-purpose numerical methods.

%For systems that are difficult to solve using general-purpose methods, 
For certain kinds of systems that are difficult to solve by general-purpose methods, structure-preserving numerical methods have been studied \cite{GIBook}. %. These methods are designed by focusing on structural properties of the system\cite{GIBook}.
However, the overall structure of coupled systems consisting of different equations can be complicated due to the differences in the properties of the individual subsystems and the effects of the way of coupling. Thus %some theoretical analysis of 
theoretical investigations of the structures of coupled systems are required.

In this study, we consider coupled systems, especially those which consist of Hamiltonian systems as their subsystems.
For Hamiltonian systems, symplectic integrators are known to be efficient~\cite{ReichBook}. 
%as a structure-preserving numerical solution method 
These methods are based on the conservation of the symplectic form of the Hamiltonian system and have good properties such as bounded energy variation and discrete versions of various conservation laws. %, and known methods and conditions for constructing explicit methods.
Therefore, if the coupled system is a Hamiltonian system, the symplectic integrators may be the best choice. %effective.
However, even if subsystems are Hamiltonian systems, the coupled system may not be Hamiltonian. %it does not necessarily mean that its coupled system is a Hamiltonian system.
It was shown that a specific coupled Hamiltonian system composed of the wave equation and the elasticity equation is Hamiltonian \cite{Terakawa2020}.
The present study is a generalization of this result.

\begin{comment}

More precisely, we have clarified the conditions under which a coupled system consisting of two finite dimensional Hamiltonian systems becomes a Hamiltonian system.
This enables us to determine the applicability of the symplectic integrators to the coupled systems.
In the case of coupled infinite-dimensional Hamiltonian systems, there are some degrees of freedom in the coupled model depending on the discretization method.
The results of this study can be used as a guideline for choosing an appropriate coupled model in such cases.

Since the purpose of this study is to perform numerical calculations, the infinite-dimensional system represented by partial differential equations is assumed to be the system after semi-discretization in the spatial direction.

In Section \ref{sec:experiment}, numerical experiments are conducted on a coupled model of an simple elastic beam and a spring-masses system to verify the effectiveness of a symplectic integrator for the coupled system.

\end{comment}

%%%%%%%%%%%%%%%%%%%%%%%%%%%%%%%%%%%%%%%%%%%%%%%%%%%
\section{Hamiltonian systems and the conservation of symplectic forms}

A Hamiltonian system is typically defined in the following way. %with an coordinate:% as follows:
\begin{dfn} \label{dfn:Hamiltonian_sys-coordinate}
    Let $(M, \omega)$ be a symplectic manifold and
    $z(t):t \mapsto z(t) \in M$ be the state variable of the system.
    If there exists a function $H(z)$ called Hamiltonian
    and a skew-symmetric matrix $S$ corresponding to the symplectic form $\omega$ %of which the corresponding 2-form is closed and nondegenerate,
    such that the time evolution of $z$ is represented in
    \begin{align}
        \frac{\d z}{\d t} = S \nabla H
        \label{eq:general-hamiltonian-sys}
        ,
    \end{align}
    the system is called a Hamiltonian system.
\end{dfn}

While the equation \eqref{eq:general-hamiltonian-sys} is typically employed, the same equation can be represented in the following coordinate-free form.
%numerical computations deal with the representation of Definition \ref{dfn:Hamiltonian_sys-coordinate} under some coordinate system, there are known geometric formulations of Hamiltonian systems that do not depend on coordinates.

\begin{dfn}
    Let $(M,\omega)$ be a symplectic manifold.
    If $X$, the vector field of the system, satisfies 
    \begin{align*}
        i_X \omega = \d H
    \end{align*}
    for a function $H:M\to \R$ and a symplectic form $\omega$, 
    the system is called a Hamiltonian system and $X$ is called a Hamiltonian vector field, where $i_X$ is the interior product and $\d$ is the exterior derivative. 
\end{dfn}

This geometric representation can be used to determine whether a system is a Hamiltonian system or not.

\begin{dfn} \label{dfn:sympl}
    If a vector field $X$ on a symplectic manifold $(M, \omega)$ satisfies
    \begin{align}
		\mathcal{L}_X \omega = 0
		\label{eq:Lie0}
		,
	\end{align}
	where $\mathcal{L}_X$ is the Lie derivative with respect to $X$,
	then $X$ is said to be symplectic.
\end{dfn}

\begin{thm}
    Hamiltonian vector fields satisfy \eqref{eq:Lie0},
    and if a vector field satisfies \eqref{eq:Lie0} then it is at least locally a Hamiltonian vector field. 
\end{thm}

For details, see \cite{SymplecticTop}. 

%%%%%%%%%%%%%%%%%%%%%%%%%%%%%%%%%%%%%%%%%%%%%%%%%%%
\section{Symplectic integrators}

Symplectic integrators are methods for discretizing symplectic flows while preserving their properties.

\begin{dfn}
    For local coordinates $(q_1, \ldots, q_m, p_1,$\\$ \ldots, p_m)$ and
	the standard symplectic form $\omega := \d q_1 \wedge \d p_1 + \cdots + \d q_m \wedge \d p_m$,
	let a vector field $X$ that defines its flow $\phi _t$ satisfy
	%\begin{align*}
		$\mathcal{L}_X \omega = 0$.
	%\end{align*}
	Then a discretized $\phi _t$
	\begin{align*}
		\Phi \approx \left. \phi_t \right|_{t = \Delta t}
	\end{align*}
	such that
	\begin{align*}
		(
			q_1^{(n+1)},
			\ldots,
			q_m^{(n+1)},
			p_1^{(n+1)},
			\ldots,
			p_m^{(n+1)}
		)^{\top}%\mathrm{T}}
		\\
		=	\Phi\ 
			(
				q_1^{(n)},
				\ldots,
				q_m^{(n)},
				p_1^{(n)},
				\ldots,
				p_m^{(n)}
			)^{\top}%\mathrm{T}}
		%\begin{pmatrix}
		%	q_1^{(n+1)} \\
		%	\vdots \\
		%	q_m^{(n+1)} \\
		%	p_1^{(n+1)} \\
		%	\vdots \\
		%	p_m^{(n+1)}
		%\end{pmatrix}
		%=	\Phi
		%	\begin{pmatrix}
		%		q_1^{(n)} \\
		%		\vdots \\
		%		q_m^{(n)} \\
		%		p_1^{(n)} \\
		%		\vdots \\
		%		p_m^{(n)}
		%	\end{pmatrix}
	\end{align*}
	and
	\begin{align*}
		\d q_1^{(n+1)} \wedge \d p_1^{(n+1)} + \cdots + \d q_m^{(n+1)} \wedge \d p_m^{(n+1)} \\
		= \d q_1^{(n)} \wedge \d p_1^{(n)} + \cdots + \d q_m^{(n)} \wedge \d p_m^{(n)} 
	\end{align*}
	is called a symplectic integrator.
\end{dfn}

A numerical solution by a symplectic integrator is considered to be symplectic in the following sense\cite{GIBook}.
The solution is a sequence of discrete points in space which can be regarded as points on a solution curve of a certain Hamiltonian equation defined by the symplectic integrator.
Such a curve is generally different from the solution $\phi_t$ of the original equation, but it is expected to preserve the symplectic form.

%shadow hamiltonian no存在定理

%%%%%%%%%%%%%%%%%%%%%%%%%%%%%%%%%%%%%%%%%%%%%%%%%%%
\section{Coupling with interaction terms}

In this study, we consider the following coupled Hamiltonian systems
%with interaction terms.
%\begin{dfn}
    %We define 
    %We suppose that the coupled system 
    that consist of Hamiltonian systems $H_1$ and $H_2$ with interaction terms $f_1$ and $f_2$:% as follows:
	\begin{align}
		\frac{\d}{\d t}
			\begin{pmatrix}
				q_1 \\ p_1 \\ q_2 \\ p_2
			\end{pmatrix}
		= 	\begin{pmatrix}
				O & I & O & O \\
				-I & O & O & O \\
				O & O & O & I \\
				O & O & -I & O 
			\end{pmatrix}
				\begin{pmatrix}
					\displaystyle \frac{\partial H_1}{\partial q_1} \\ \\[-0.8em]
					\displaystyle \frac{\partial H_1}{\partial p_1} \\ \\[-0.8em]
					\displaystyle \frac{\partial H_2}{\partial q_2} \\ \\[-0.8em]
					\displaystyle \frac{\partial H_2}{\partial p_2}
				\end{pmatrix}
			+	\begin{pmatrix}
					0 \\ f_1 \\ 0 \\ f_2
				\end{pmatrix}.
		\label{eq:general-coupled-sys}
	\end{align}
	The associated vector field $X$ is also defined as the right-hand side of \eqref{eq:general-coupled-sys}. Note that $q_1, q_2, p_1, p_2, f_1, f_2$ are not necessarily scalars; they can be vectors. 
	\begin{comment}
	\begin{align}
		X := 
			\frac{\d}{\d t}
				\begin{pmatrix}
					q_1 \\ p_1 \\ q_2 \\ p_2
				\end{pmatrix}.
	\end{align}
	\end{comment}
%\end{dfn}

The specific form of the interaction is problem-specific and it determines whether the coupled system is a Hamiltonian system or not.
If the system represented by \eqref{eq:general-coupled-sys} is transformed into \eqref{eq:general-hamiltonian-sys}, then the system is Hamiltonian; however, it is in general difficult to check whether such a transformation exists. This study provides conditions to determine whether the coupled system is a Hamiltonian system for a given interaction.

%%%%%%%%%%%%%%%%%%%%%%%%%%%%%%%%%%%%%%%%%%%%%%%%%%%
%\section{Derivation of the condition}
\section{Main result}

\begin{lem} \label{lem:Lie}
    Let the standard symplectic form $\omega$ of the system \eqref{eq:general-coupled-sys} be
    %\begin{align*}
	$
		\omega :=
			\d q_1 \wedge \d p_1
			+ \d q_2 \wedge \d p_2
			.
	$
	%\end{align*}
    The Lie derivative $\mathcal{L}_X \omega$ of $\omega$ with respect to $X$ is
	\begin{align*}
		\mathcal{L}_X \omega
		=	\d f_1 \wedge \d q_1
			+ \d f_2 \wedge \d q_2
	    .
	\end{align*}
\end{lem}

\begin{proof}
    From the Cartan formula, it holds that
    \begin{align*}
        \mathcal{L}_X \omega = \d(i_X(\omega)) + i_X(\d \omega).
    \end{align*}
    $i_X(\d \omega)$ is always zero because $\omega$ is closed. % due to the closedness of $\omega$.
    $i_X(\omega)$ in the first term of the right-hand side is
%    \begin{align*}
%        &i_X(\omega) \\
%        &=	\sum_i \left(
%			\frac{\d (p_1)_i}{\d t} \d (q_1)_i - \frac{\d (q_1)_i}{\d t} \d (p_1)_i
%		\right)
%		\\ &\quad
%		+ \sum_j \left(
%			\frac{\d (p_2)_j}{\d t} \d (q_2)_j - \frac{\d (q_2)_j}{\d t} \d (p_2)_j
%		\right)
%	\\
%	&=	\sum_i \left(
%			\left( - \frac{\partial H_1}{\partial (q_1)_i} + (f_1)_i \right) \d (q_1)_i
%			- \frac{\partial H_1}{\partial (p_1)_i} \d (p_1)_i 
%		\right)
%		\\ &\quad
%		+	\sum_j \left(
%				\left( - \frac{\partial H_2}{\partial (q_2)_j} + (f_2)_j \right) \d (q_2)_j
%				- \frac{\partial H_2}{\partial (p_2)_j} \d (p_2)_j
%			\right)
%    \end{align*}
    \begin{align*}
        i_X (\omega)
		&=	\frac{\d p_1}{\d t} \d q_1 - \frac{\d q_1}{\d t} \d p_1
			+ \frac{\d p_2}{\d t} \d q_2 - \frac{\d q_2}{\d t} \d p_2
		\\
		&=	\left( - \frac{\partial H_1}{\partial q_1} + f_1 \right) \d q_1
			- \frac{\partial H_1}{\partial p_1} \d p_1
			%\\ &\quad
			+ \left( - \frac{\partial H_2}{\partial q_2} + f_2 \right) \d q_2
			- \frac{\partial H_2}{\partial p_2} \d p_2
	\end{align*}
	and hence
	\begin{align*}
		\d (&i_X (\omega)) \\
		&=	\d \left( - \frac{\partial H_1}{\partial q_1} + f_1 \right) \wedge \d q_1
			- \d \left( \frac{\partial H_1}{\partial p_1} \right) \wedge \d p_1
			%\\ &\quad
			+ \d \left( - \frac{\partial H_2}{\partial q_2} + f_2 \right) \wedge \d q_2
			- \d \left( \frac{\partial H_2}{\partial p_2} \right) \wedge \d p_2
		\\
		&=	\left(
				- \frac{\partial^2 H_1}{\partial^2 q_1} \d q_1
				- \frac{\partial^2 H_1}{\partial q_1 \partial p_1} \d p_1
				+ \d f_1 
			\right) \wedge \d q_1
			- \left(%\\ &\quad- \left(
				\frac{\partial^2 H_1}{\partial p_1 \partial q_1} \d q_1
				+ \frac{\partial^2 H_1}{\partial^2 p_1} \d p_1
			\right) \wedge \d p_1
			\\ &\quad + \left(
				- \frac{\partial^2 H_2}{\partial^2 q_2} \d q_2
				- \frac{\partial^2 H_2}{\partial q_2 \partial p_2} \d p_2
				+ \d f_2
			\right) \wedge \d q_2
			- \left(%\\ &\quad - \left(
				\frac{\partial^2 H_2}{\partial p_2 \partial q_2} \d q_2
				+ \frac{\partial^2 H_2}{\partial^2 p_2} \d p_2
			\right) \wedge \d p_2
		\\
		&= \d f_1 \wedge \d q_1 + \d f_2 \wedge \d q_2,
    \end{align*}
    which proves this lemma. 
\end{proof}
Lemma \ref{lem:Lie} gives the condition for a coupled system to preserve the symplectic form.

\begin{thm} \label{thm:main-result}
	If the coupled system referred to in Lemma \ref{lem:Lie} satisfies
	\begin{align*}
		\d f_1 \wedge \d q_1 + \d f_2 \wedge \d q_2 = 0,
	\end{align*}
	then it preserves the symplectic form $\omega$.
\end{thm}

\begin{proof}
    It follows immediately from Lemma \ref{lem:Lie}.
\end{proof}

\begin{rmk}
An important fact seen from Theorem \ref{thm:main-result} is that the symplectic form for the coupled system is the direct sum of the symplectic forms for the subsystems. 
%consists of the subsystems' forms in the simply added terms.
This modularity makes it easy to couple additional subsystems one after another.
\end{rmk}

%%%%%%%%%%%%%%%%%%%%%%%%%%%%%%%%%%%%%%%%%%%%%%%%%%%
\section{Numerical experiments} \label{sec:experiments}

In this section, we consider %deal with 
a composition of a simple elastic beam and a spring-mass system.
We determine the interaction without considering symplecticity, then verify it using Theorem \ref{thm:main-result}.

Let the coordinates be as illustrated in Fig. \ref{fig:schematic}. The %continuous 
equation of the beam is
\begin{align*}
    &\rho A u_{tt} = - EI u_{xxxx}, \\
    &u(0,t)=u(L,t)=u_{xx}(0,t)=u_{xx}(L,t)=0
    ,
\end{align*}
where $\rho$ is the density, $A$ is the cross-sectional area, $E$ is the elastic modulus and $I$ is the second moment of area. We suppose that all these quantities are constants.
\begin{figure}[tb]
    \centering
    \includegraphics[width=0.6\linewidth]{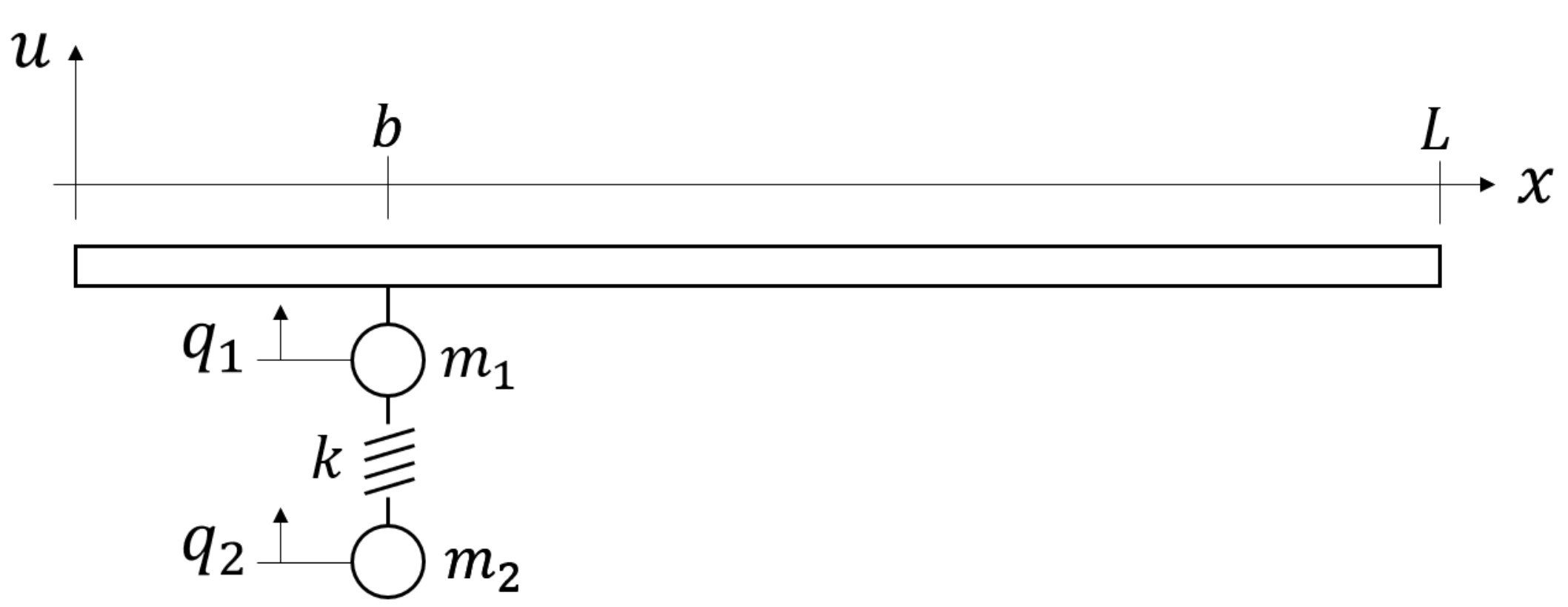}
    \caption{Schematic of the coupled system. It consists of a simple elastic beam and a mass-spring system, and a point on the beam and the mass $m_1$ are fixed together.}
    \label{fig:schematic}
\end{figure}

We consider %treat 
a %continuous 
coupled system: 
\begin{gather*}
    \frac{\d}{\d t}
		\begin{pmatrix}
			u \\ v \\ q_1 \\ q_2 \\ p_1 \\ p_2
		\end{pmatrix}
	    = 	J
    		\begin{pmatrix}
    			%\frac{v}{\rho A} \\
    			v / \rho A \\
    			- EI u_{xxxx} \\
    			p_1 m_1 \\
    			p_2 m_2 \\
    			k (q_2 - q_1) \\
    			-k (q_2 - q_1)
    		\end{pmatrix}
		+	\begin{pmatrix}
				0 \\ f \delta(x - b) \\ 0 \\ 0 \\ - f \\ 0
			\end{pmatrix}
    ,
	\\
	J   =   \begin{pmatrix}
	            \begin{matrix}
	                0 & 1 \\
    			    -1 & 0 \\
	            \end{matrix}
    			& O \\
    			O &
    			\begin{matrix}
    			    O & I_2 \\
    			    -I_2 & O
    			\end{matrix}
    		\end{pmatrix}
    ,\quad
    I_2 =   \begin{pmatrix}
                1 & 0 \\
                0 & 1
            \end{pmatrix}
	%J   =   \begin{pmatrix}
    %			0 & 1 & 0 & 0 & 0 & 0 \\
    %			-1 & 0 & 0 & 0 & 0 & 0 \\
    %			0 & 0 & 0 & 0 & 1 & 0 \\
    %			0 & 0 & 0 & 0 & 0 & 1 \\
    %			0 & 0 & -1 & 0 & 0 & 0 \\
    %			0 & 0 & 0 & -1 & 0 & 0
    %		\end{pmatrix}
\end{gather*}
and suppose a corresponding discrete system is given as
\begin{align}
    &\frac{\d}{\d t}
		\begin{pmatrix}
			u_i \\ v_i \\ q_1 \\ q_2 \\ p_1 \\ p_2
		\end{pmatrix}
	    = 	J
    		\begin{pmatrix}
    		    v_i / \rho A \\
    			- EI \delta_4 u_i \\
    			p_1 m_1 \\
    			p_2 m_2 \\
    			k (q_2 - q_1) \\
    			-k (q_2 - q_1)
    		\end{pmatrix}
		+	\begin{pmatrix}
				0 \\ F_i \\ 0 \\ 0 \\ - f \\ 0
			\end{pmatrix}
	,
	\label{eq:disc-sys}
	\\
	&\delta_4 u_i = \frac{u_{i+2} - 4u_{i+1} + 6u_i - 4u_{i-1} + u_{i-2}}{\Delta x^4}, \notag
	\\
	&(1\leq i \leq N_x-1), \notag
	\\
	&u_0(t) = u_{N_x}(t) = 0, \label{eq:discrete-bc-displacement} \\
	&u_{-1}(t) = -u_1(t),\ u_{N_x-1}(t) = -u_{N_x}(t) \label{eq:discrete-bc-bending}
	.
\end{align}
The discrete boundary conditions \eqref{eq:discrete-bc-displacement} and \eqref{eq:discrete-bc-bending} are derived as follows.
\eqref{eq:discrete-bc-displacement} is understood as discretized $u(0,t)=u(L,t)=0$.
%For \eqref{eq:discrete-bc-bending}, consider
%the second order differential corresponding to $\delta_4$ at $u_1$ (and $u_{N_x-1}$):
\eqref{eq:discrete-bc-bending} is derived from the discrete version of the boundary condition $u_{xx} = 0$, e.g., for $u_{xx}(0, t) = 0$, 
\begin{align}
    \frac{u_{-1} - 2u_0 + u_1}{\Delta x^2} = 0. \label{eq:disc-uxx=0}
\end{align}
The boundary condition \eqref{eq:discrete-bc-bending} is obtained by combining \eqref{eq:discrete-bc-displacement} and \eqref{eq:disc-uxx=0}. %, the discrete version of $u_{xx}(0,t)=0$, you get the other boundary condition 

The two subsystems are coupled at $u_{n_b}$ on the semi-discretized beam and $q_1$ with an interaction vector:
\begin{align*}
    %f_i = &\begin{cases}
    %    f &(i=1) \\
    %    0 &(i=2)
    %\end{cases}
    %, \quad
    F_i = \begin{cases}
        f/\Delta x &(i=n_b) \\
        0 & (i\neq n_b)
    \end{cases}
    .
\end{align*}

Each subsystem is a Hamiltonian system under the following Hamiltonians:
\begin{align*}
    H_{\rm{el}}
    %&=  \sum_{i=1}^{Nx-1} \frac{1}{2} v_i^2 \Delta x
    %    + \sum_{i=1}^{Nx-1} \frac{1}{2} \left( \frac{u_{i+1} - 2u_i + u_{i-1}}{\Delta x^2} \right)^2 \Delta x
    &=  \sum_{i=1}^{N_x-1} \frac{1}{2} \left(
            \frac{v_i^2}{\rho A} 
            + EI \left( \frac{u_{i+1} - 2u_i + u_{i-1}}{\Delta x^2} \right)^2
        \right) \Delta x
    ,
    \\
    H_{\rm{sp}}
    &=  \frac{p_1^2}{2m_1}
        + \frac{p_2^2}{2m_2}
        + \frac{k}{2}(q_2 - q_1)^2
    .
\end{align*}
Hence the discrete coupled system is written with these Hamiltonians:
\begin{gather*}
    \frac{\d}{\d t}
		\begin{pmatrix}
			u_i \\ v_i \\ q_j \\ p_j
		\end{pmatrix}
	= 	\begin{pmatrix}
			0 & 1 & 0 & 0 \\
			-1 & 0 & 0 & 0 \\
			0 & 0 & 0 & 1 \\
			0 & 0 & -1 & 0
		\end{pmatrix}
			\begin{pmatrix}
				\displaystyle \frac{\delta H_{\rm{el}}}{\delta u_i} \\ \\[-0.8em]
				\displaystyle \frac{\delta H_{\rm{el}}}{\delta v_i} \\ \\[-0.8em]
				\displaystyle \frac{\partial H_{\rm{sp}}}{\partial q_j} \\ \\[-0.8em]
				\displaystyle \frac{\partial H_{\rm{sp}}}{\partial p_j} \\ \\[-0.8em]
			\end{pmatrix}
		+	\begin{pmatrix}
				0 \\ F_i \\ 0 \\ f_j \\ 0
			\end{pmatrix}
    ,
    \\
	f_j
	=  \begin{cases}
            - f &(j=1) \\
            0 &(j=2)
        \end{cases}
\end{gather*}
where $\displaystyle \frac{\delta H_{\rm{el}}}{\delta u_i} = - EI \delta_4 u_i$ and $\displaystyle \frac{\delta H_{\rm{el}}}{\delta v_i} = \frac{v_j}{\rho A}$ are essentially the discrete variational derivatives proposed by Furihata and Matsuo~\cite{FurihataBook}.

$f$, the magnitude of the interaction, is determined using an additional assumption.
The coupling of the target system seems not to include any dissipative component, therefore we suppose the total energy $H := H_{\rm{sp}} + H_{\rm{el}}$ to be conserved.
The time derivative of the total energy is
\begin{align*}
    \frac{\d H}{\d t}
    &=  \frac{\delta H}{\delta u_i} \frac{\d u_i}{\d t} \Delta x
        +   \frac{\delta H}{\delta v_i} \frac{\d v_i}{\d t} \Delta x
        %\\ &\quad
        +   \frac{\partial H}{\partial q_j} \frac{\d q_j}{\d t}
        +   \frac{\partial H}{\partial p_j} \frac{\d p_j}{\d t}
    \\
    &=  \frac{\delta H}{\delta u_i} \frac{\delta H}{\delta v_i} \Delta x
        +   \frac{\delta H}{\delta v_i} \left(- \frac{\delta H}{\delta u_i} + F_i\right) \Delta x
        %\\ &\quad
        +   \frac{\partial H}{\partial q_j} \frac{\partial H}{\partial p_j}
        +   \frac{\partial H}{\partial p_j} \left(- \frac{\partial H}{\partial q_j} + f_j\right)
    \\
    &=  \left( \frac{v_{n_b}}{\rho A} - \frac{p_1}{m_1} \right) f.
    %=0
    %\\
    %&=  0
    %.
\end{align*}
Therefore, if $\displaystyle \frac{p_1(t)}{m_1} = \frac{v_{n_b}(t)}{\rho A}$ holds, the total energy will be conserved.
This sufficient condition is equivalent to
\begin{align}
    \frac{1}{\rho A} v_{n_b}(0) - \frac{1}{m_1} p_1(0) = 0 \label{eq:coupling-initial_cond},\\
    \frac{1}{\rho A} \frac{\d v_{n_b}}{\d t}(t) - \frac{1}{m_1} \frac{\d p_1}{\d t}(t) = 0 \label{eq:coupling-dt_cond}
    .
\end{align}
Substituting \eqref{eq:disc-sys} into \eqref{eq:coupling-dt_cond}, we can determine $f$:
\begin{gather*}
    \frac{1}{\rho A} \left(EI \delta_4 u_{n_b} - \frac{f}{\Delta x}\right)
    -   \frac{1}{m_1} (k(q_2 - q_1) + f) = 0
    \notag \\
    \iff
    f = - \frac{\rho A \Delta x m_1}{\rho A \Delta x + m_1} \left(
            \frac{EI}{\rho A} \delta_4 u_{n_b} + \frac{k}{m_1}(q_2 - q_1)
        \right)
    .
\end{gather*}

Now the entire semi-discretized coupled system is described.
However, the initial condition \eqref{eq:coupling-initial_cond} and  $f$ given by \eqref{eq:coupling-dt_cond} only guarantee the conservation of the total energy.
In other words, we should check the symplecticity of the coupled system using Theorem \ref{thm:main-result}.

\begin{figure}[t]
    \centering
    \includegraphics[width=0.6\linewidth]{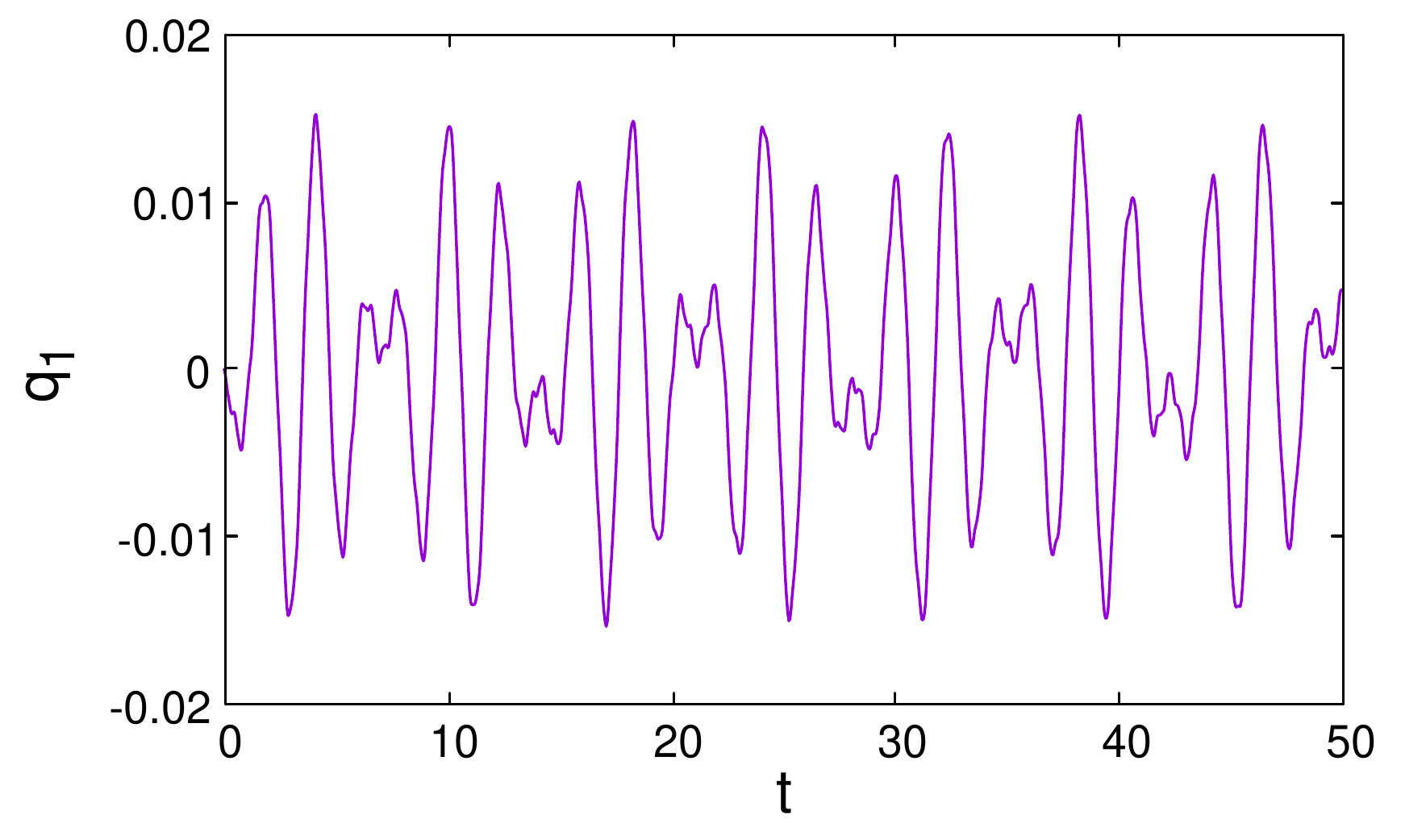}
    \caption{The displacement of the coupling point.}
    \label{fig:displacements}
\end{figure}

\begin{figure}[t]
    \centering
    \includegraphics[width=0.6\linewidth]{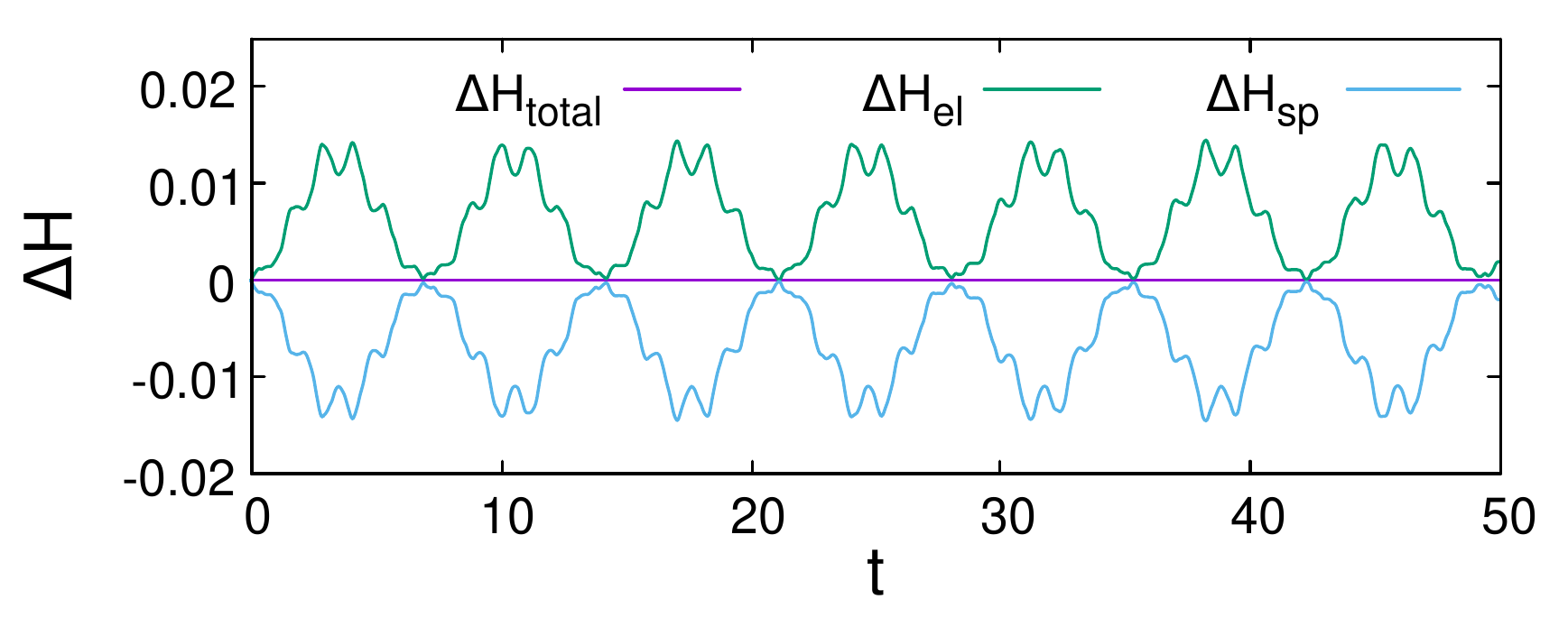}
    \caption{The energy variation of the subsystems and their sum.}
    \label{fig:energy-sum}
\end{figure}

The relationship between $\d u_{n_b}$ and $\d q_1$ is obtained by integrating and taking the exterior derivative of \eqref{eq:coupling-dt_cond}:
\begin{align*}
    \frac{1}{\rho A} v_{n_b} - \frac{1}{m_1} p_1 &= C \quad (C:\mathrm{constant}) \\
    u_{n_b} - q_1 &= C \\
    \d u_{n_b} - \d q_1 &= \d C = 0
    .
\end{align*}
Hence the condition of Theorem \ref{thm:main-result}
\begin{align*}
    \d f \wedge u_{n_b} - \d f \wedge \d q_1 = 0
\end{align*}
holds, and the coupled system is symplectic.

We conducted a numerical experiment with the S\"ormer--Verlet method %the symplectic Euler method
under $L=1, N_x=51, b=0.2, n_b=10, T=50, N_t=10^5, \rho=10, A=E=I=1, m_1=m_2=0.1, k=0.5$. The spatial and temporal step sizes are $\Delta x = 2\times10^{-2}, \Delta t = 5\times10^{-4}$.
At $t=0$, the displacements of the beam and the coupled mass $m_1$ are set to 0, while $q_2$ is set to $-1.0$.

The results are shown in Figs. \ref{fig:displacements}--\ref{fig:long-energy}. %, \ref{fig:energy-sum} and \ref{fig:long-energy}.
Fig. \ref{fig:energy-sum} shows that the energy variations from its initial values for the subsystems ($\Delta H_{\mathrm{el}}$ and $\Delta H_{\mathrm{sp}}$) are complementary to each other, and the total energy $H$ is conserved.

The symplecticity can be confirmed by %seems to appear as 
the order property of the modified Hamiltonian \cite{ReichBook} and the S\"ormer-Verlet method. %the symplectic Euler method.
To check the order property of the method, we conducted additional experiments with different $\Delta t$'s under fixed $T=500$.
Fig. \ref{fig:long-energy} shows the results of $\Delta t = 5\times10^{-4}$ and $\Delta t = 2.5\times10^{-4}$.
%We can see that the fluctuation of the total energy is proportional to the time step.
%It is thought that the proportionality reflects the fact that method preserves the modified Hamiltonian in order 1, and that the symplectic integrators are valid for the target system.
We can see that the variation of the total energy decreases in proportion to the square of the time step.
%It is thought that 
This proportionality reflects the fact that the method preserves the modified Hamiltonian in order 1, and that the symplectic integrators are valid for the system.

\begin{figure}[t]
    \centering
    \includegraphics[width=0.6\linewidth]{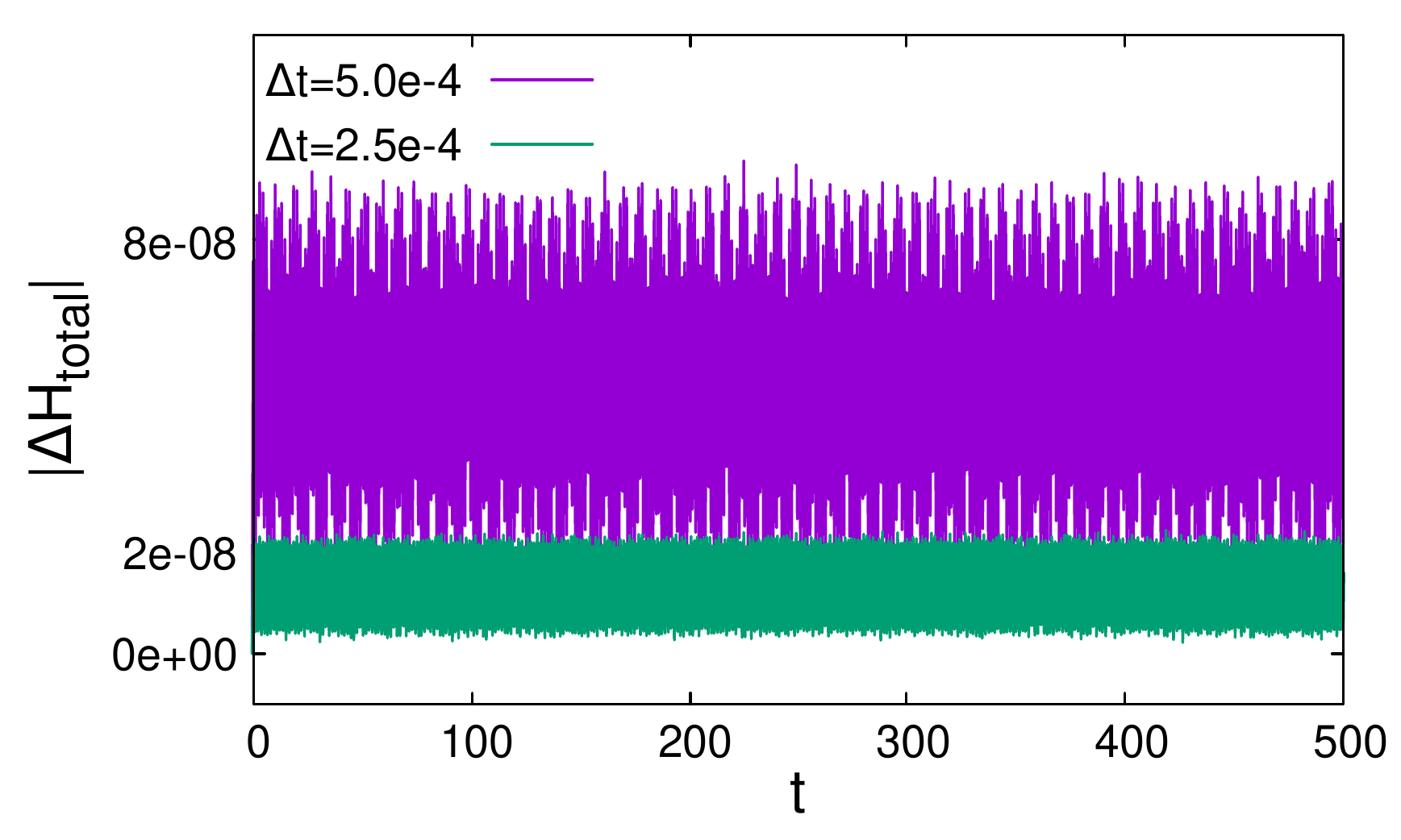}
    %\caption{The energy fluctuations for different time steps. Halving the time step also halves the energy fluctuation; it shows the order property of the symplectic Euler method.}
    \caption{Variation of the total energy for different time steps. Halving the time step quarters the energy fluctuation; it shows the order property of the S\"ormer-Verlet method.}
    \label{fig:long-energy}
\end{figure}

%%%%%%%%%%%%%%%%%%%%%%%%%%%%%%%%%%%%%%%%%%%%%%%%%%%
\section{Conclusion}

We have proposed a condition under which a coupled system that consists of two Hamiltonian systems became locally a Hamiltonian system.
This result enables us to check the applicability of the symplectic integrators to complicated coupled systems.%, and exploit the useful properties of the integrators for such difficult systems.

%As a numerical experiment, the S\"ormer-Verlet method %the symplectic Euler method 
%is applied to a coupled system consisting of an elastic beam and spring-mass subsystem, showing that the conservation of the total energy and the order property of the method is observed as expected.

As related work, %As an extension of Hamiltonian systems, 
port-Hamiltonian systems, which are an extension of Hamiltonian systems,  have been studied\cite{PHBook}.
It is known that the composition of a port-Hamiltonian system is also a port-Hamiltonian system; % as a basic result.
however, port-Hamiltonian systems are formulated by focusing on the Dirac structure rather than the conservation of symplectic forms, and the applicability of symplectic integrators is not well understood. %always clear. 
The availability of symplectic integrators for such systems should be investigated in future work.

\section*{Acknowledgment}
This work is supported by JST CREST Grant Number JPMJCR1914.

\bibliographystyle{unsrtnat}
%\bibliography{references}  %%% Uncomment this line and comment out the ``thebibliography'' section below to use the external .bib file (using bibtex) .

%%% Uncomment this section and comment out the \bibliography{references} line above to use inline references.
% \begin{thebibliography}{1}

% 	\bibitem{kour2014real}
% 	George Kour and Raid Saabne.
% 	\newblock Real-time segmentation of on-line handwritten arabic script.
% 	\newblock In {\em Frontiers in Handwriting Recognition (ICFHR), 2014 14th
% 			International Conference on}, pages 417--422. IEEE, 2014.

% 	\bibitem{kour2014fast}
% 	George Kour and Raid Saabne.
% 	\newblock Fast classification of handwritten on-line arabic characters.
% 	\newblock In {\em Soft Computing and Pattern Recognition (SoCPaR), 2014 6th
% 			International Conference of}, pages 312--318. IEEE, 2014.

% 	\bibitem{hadash2018estimate}
% 	Guy Hadash, Einat Kermany, Boaz Carmeli, Ofer Lavi, George Kour, and Alon
% 	Jacovi.
% 	\newblock Estimate and replace: A novel approach to integrating deep neural
% 	networks with existing applications.
% 	\newblock {\em arXiv preprint arXiv:1804.09028}, 2018.

% \end{thebibliography}

\end{document}